\documentclass[12pt, reqno]{amsart}
\usepackage{mathrsfs}
\usepackage{amsmath, amstext, amsbsy, amssymb, amscd,mathrsfs}
\usepackage{amsmath}
\usepackage{amsxtra}
\usepackage{amscd}
\usepackage{amsthm}
\usepackage{amsfonts}
\usepackage{amssymb}
\usepackage{eucal}
\usepackage{color}

\input postpictex
\input xy

\xyoption{all}

\newtheorem{theorem}{Theorem}[section]
\newtheorem{lemma}[theorem]{Lemma}
\newtheorem{proposition}[theorem]{Proposition}

\theoremstyle{definition}

\newtheorem{remark}[theorem]{Remark}

\numberwithin{equation}{section}

\newcommand{\g}{\mathfrak{g}}
\newcommand{\n}{\mathfrak{n}}
\newcommand{\h}{\mathfrak{h}}
\newcommand{\p}{\mathfrak{p}}

\newcommand{\m}{\mathfrak{m}}

\newcommand{\rea}[2]{U_{#1}(#2)} 
\newcommand{\raa}[3]{U_{#1,#2}(#3)} 
\newcommand{\rsa}[2]{S_{#1}(#2)} 
\newcommand{\Fp}{\mathscr{F}(\mathfrak{g},\mathfrak{p})}
\newcommand{\X}[2]{\mathscr{X}_{#1,#2}} 
\newcommand{\Y}[2]{\mathscr{Y}_{#1,#2}} 
\newcommand{\Hom}{\text{Hom}}
\newcommand{\pth}[1]{{#1}^{[p]}} 
\newcommand{\ra}{\rightarrow}
\newcommand{\sudim}{\underline{\text{dim}}\,}
\newcommand{\ev}[1]{{#1}_{\bar{0}}}
\newcommand{\od}[1]{{#1}_{\bar{1}}}

\newcommand{\osp}{\mathfrak{osp}}
\newcommand{\Der}{\text{Der}}
\newcommand{\codim}{\text{codim}\,}
\newcommand{\scodim}{\underline{\text{codim}}\,}

\newcommand{\la}{\lambda}
\newcommand{\La}{\Lambda}

\newcommand{\End}{\text{End}}

\newcommand{\gl}{{\mathfrak{gl}}}

\newcommand{\Ad}{\text{Ad}}

\newcommand{\Z}{ \mathbb Z }

{\vskip-\lastskip\medskip
  \noindent
  {\em #1.}\enspace
  }%
{\qed\par\medskip
  }

\makeatletter
\def\Ddots{\mathinner{\mkern1mu\raise\p@
\vbox{\kern7\p@\hbox{.}}\mkern2mu
\raise4\p@\hbox{.}\mkern2mu\raise7\p@\hbox{.}\mkern1mu}}
\makeatother

\begin{document}
\title[Lie Superalgebras in Prime Characteristic]
{Representations of Lie Superalgebras in Prime Characteristic III}

\author[Lei Zhao]{Lei Zhao}
\address{Department of Mathematics, University of Virginia,
Charlottesville, VA 22904} \email{lz4u@virginia.edu}

\begin{abstract}
For a restricted Lie superalgebra $\g$ over an algebraically closed
field of characteristic $p > 2$, we generalize the deformation
method of Premet and Skryabin to obtain results on the $p$-power and
$2$-power divisibility of dimensions of $\g$-modules. In particular,
we give a new proof of the Super Kac-Weisfeiler conjecture for basic
classical Lie superalgebras. The new proof allows us to improve
optimally the assumption on $p$. We also establish a semisimplicity
criterion for the reduced enveloping superalgebras associated with
semisimple $p$-characters for all basic classical Lie superalgebras
using the technique of odd reflections.
\end{abstract}

\maketitle
\date{}
  \setcounter{tocdepth}{1}

\section{Introduction}

\subsection{}
In \cite{WZ1}, Wang and the author initiated the study of modular
representation theory of Lie superalgebras over an algebraically
closed field $K$ of characteristic $p>2$. Among other things, a
superalgebra generalization (called Super KW conjecture) of
celebrated Kac-Weisfeiler conjecture (Premet's Theorem) was
formulated; and we established it for the most important class of
Lie superalgebras---the basic classical Lie superalgebras, which are
first classified over the complex numbers by Kac \cite{Kac} and
Scheunert-Nahm-Rittenberg \cite{SNR}. Our work generalized the
earlier work on Lie algebras of reductive algebraic groups by
Kac-Weisfeiler \cite{KW}, Friedlander-Parshall \cite{FP}, Premet
\cite{Pr1,Pr2}, Skryabin \cite{Skr}, and others (see Jantzen
\cite{Jan1} for an excellent review and extensive references on
modular representations of Lie algebras).

In our proof of the super KW conjecture, a $\Z$-grading of the basic
classical Lie superalgebras plays an important role. In order to
obtain the grading, we imposed somewhat restrictive conditions on
$p$ \cite[Section~2.2]{WZ1}.

In \cite{PS}, Premet and Skryabin developed deformation techniques
by considering a family of $\mathscr L$-associative algebras for a
restricted Lie algebra $\mathscr L$ to derive results on dimensions
of simple $\mathscr L$-modules. In particular, their method gives a
new proof of the Kac-Weisfeiler conjecture which differs completely
from Premet's original approach \cite{Pr1}.

\subsection{}
The first main goal of this paper is to generalize some of the ideas
in \cite{PS} to the superalgebra setting. In particular, we provide
a new proof of Super KW conjecture for basic classical Lie
superalgebras so that the over-restrictive assumption on $p$ in
\cite[Section~2.2]{WZ1} is relaxed optimally.

Our second goal is to give a simplicity criterion for baby Verma
modules as well as a semisimplicity criterion for reduced enveloping
superalgebras of basic classical Lie superalgebras with semisimple
$p$-characters.

\subsection{}
Let $\g = \ev\g+\od\g$ be an $(n_0|\,n_1)$-dimensional restricted
Lie superalgebra over $K$ and let $\xi \in \ev\g^*$. Let $S(\g)$ be
the symmetric superalgebra on $\g$. The {\em reduced symmetric
superalgebra} $\rsa{\xi}{\g}$ associated with $\xi$ is defined to be
the quotient of $S(\g)$ by the ideal generated by elements of the
form $(x-\xi(x))^p$ with $x \in \ev\g$. It is a local
(super)commutative superalgebra of dimension $p^{n_0}2^{n_1}$. Let
$\rea{\xi}{\g}$ be the reduced enveloping superalgebra as usual.

Following \cite{PS}, we introduce a family of associative
superalgebras $\raa{\xi}{\la}{\g}$, where $\xi \in \ev\g^*$ and $\la
\in K$, parametrized by the points of the projective space
$\mathbb{P}(\ev\g^*\oplus K)$ (the superalgebras
$\raa{t\xi}{t\la}{\g}$ with $t \in K^\times$ being isomorphic). The
Lie superalgebra $\g$ acts on each $\raa{\xi}{\la}{\g}$ as
derivations. The family relates the reduced enveloping superalgebra
$\rea{\xi}{\g}$ ($= \raa{\xi}{1}{\g}$) to the reduced symmetric
superalgebra $\rsa{\xi}{\g}$ ($= \raa{\xi}{0}{\g}$). As in the Lie
algebra case \cite{PS}, the reduced symmetric superalgebra
$\rsa{\xi}{\g}$ has favorable structures of $\g$-invariant ideals
(cf. Proposition~\ref{prop:sym}).

Following \cite{PS} but with slight modification, we use the method
of associated cones in invariant theory to obtain some results on
the $(p,2)$-divisibility of dimensions of $\g$-modules. In
particular, we show that (Theorem~\ref{thm-div} (ii)) for an
arbitrary restricted Lie superalgebra $\g$ and $\chi \in \ev\g^*$,
if
\begin{itemize}
\item[($\star$)] all nonzero scalar multiples of $\chi$ are conjugate under the
group $G(\ev\g)$ of automorphisms of $\ev\g$ which preserve the
restricted structure,
\end{itemize}
then the super KW conjecture holds for $\rea{\chi}{\g}$. Note that
($\star$) is a non-super condition. Now if $\g$ is one of the basic
classical Lie superalgebras as in Section~\ref{subsec:bcls} with the
optimal assumption on $p$ or the queer Lie superalgebra as in
\cite{WZ2} and if $\chi \in \ev\g^*$ is nilpotent, then condition
($\star$) is satisfied (\cite[Sections~2.8, 2.10]{Jan2}). Thus the
super KW conjecture for basic classical Lie superalgebras and the
queer Lie superalgebra with nilpotent $p$-characters holds. Together
with the Morita equivalence theorem \cite[Theorem~5.2]{WZ1}, this
gives a new proof of the super KW conjecture for basic classical Lie
superalgebras in full generality with the optimal assumption on $p$.

\subsection{} For the reduced enveloping superalgebras of basic
classical Lie superalgebras with semisimple $p$-characters, we give
a simplicity criterion for baby Verma modules. As a consequence, we
obtain a semisimplicity criterion for the reduced enveloping
superalgebras. These results, first announced in
\cite[Remark~4.5]{Z}, generalize results of Rudakov \cite{Rud} and
Friedlander-Parshall \cite{FP} for Lie algebras.

A major complication in the super case is due to the existence of
non-conjugate sets of simple roots. We settle the problem by using
the technique of odd reflections (see \cite{Ser} for example). This
approach is quite different from the proof of the corresponding
results for type I basic classical Lie superalgebras in \cite{Z}.

In his paper \cite{Zh}, C. Zhang independently stated the simplicity
criterion for baby Verma modules with semisimple $p$-characters for
basic classical Lie superalgebras (the statement of
Theorem~\ref{thm:irred-verma}). However, his proof, which relied
essentially on an erroneous lemma \cite[Lemma~3.6]{Zh}, is
incorrect.

\subsection{} The paper is laid out as follows. In Section~\ref{sec:basic-family}, after reviewing some
basic facts about modular representations of Lie superalgebras and
basic classical Lie superalgebras, we introduce the super
generalization of families of associative algebras following
\cite{PS}. Then we study the properties of invariant ideals of the
reduced symmetric superalgebras. The new proof of super KW
conjecture for basic classical Lie superalgebras is given in
Section~\ref{sec:SKW}. Finally, Section~\ref{sec:ss} is devoted to
the study of basic classical Lie superalgebras with semisimple
$p$-characters.


\noindent {\bf Acknowledgments.} The author is very grateful to his
advisor, Weiqiang Wang, for valuable suggestions and advice. He is
deeply indebted to A. Premet and S. Skryabin for their influential
ideas. The author also thanks I. Gordon for helpful discussions.

\section{Restricted Lie superalgebras and families of
$\g$-superalgebras}\label{sec:basic-family}
\subsection{} Throughout we work with an algebraically
closed field $K$ with characteristic $p > 2$ as the ground field. We
exclude $p=2$ since in that case Lie superalgebras coincide with Lie
algebras.

A superspace is a $\Z_2$-graded vector space $V = \ev{V} \oplus
\od{V}$, in which we call elements in $\ev{V}$ (resp. $\od{V}$) even
(resp. odd). Write $|v| \in \Z_2$ for the parity (or degree) of $v
\in V$, which is implicitly assumed to be ($\Z_2$-)homogeneous. A
bilinear form $f$ on $V$ is {\em supersymmetric} if $f(u,v)
=(-1)^{|u||v|} f(v,u)$ for all homogeneous $u,v \in V$. We will use
the notation
$$
\sudim V = \dim \ev{V} | \dim \od V;\qquad \dim V =\dim \ev V + \dim
\od V.
$$
If $W$ is a subsuperspace of $V$, denote
\[
\scodim_{V}W = \sudim V- \sudim W; \qquad \codim_{V} W= \dim V -\dim
W.
\]
Sometimes we simply write $\scodim W$ and $\codim W$ for short when
the total space $V$ is clear from the context.

All Lie superalgebras $\g$ will be assumed to be finite dimensional.
We will use $U(\g)$ to denote its universal enveloping superalgebra.

According to Walls \cite{W}, the finite-dimensional simple
associative superalgebras over $K$ are classified into two types:
besides the usual matrix superalgebra (called type $M$) there are in
addition simple superalgebras of type $Q$.

By vector spaces, derivations, subalgebras, ideals, modules,
submodules, and commutativity, etc. we mean in the super sense
unless otherwise specified.

For a real number $a$, we use $\lfloor a \rfloor$ to denote its
least integer upper bound, and use $\lceil a \rceil$ to denote its
greatest integer lower bound.

\subsection{}
%
%
%
Recall a restricted Lie superalgebra $\g = \ev \g \oplus \od \g$ is
a Lie superalgebra whose even subalgebra $\ev \g$ is a restricted
Lie algebra which admits a $[p]$th power map $\pth{} : \ev \g \ra
\ev \g$ satisfying certain conditions (\cite[Chap.~V]{Jac}), and
whose odd part $\od \g$ is a restricted module by the adjoint action
of the even subalgebra $\ev \g$.

All the Lie (super)algebras in this paper will be assumed to be
restricted.

Let $\g$ be a restricted Lie superalgebra, for each $\chi \in \ev
\g^*$, the {\em reduced enveloping superalgebra} of $\g$ with the
$p$-character $\chi$ is by definition the quotient of $U(\g)$ by the
ideal $I_{\chi}$ generated by all $x^p -\pth x -\chi(x)^p$ with $x
\in \ev \g$.

We further recall the definition of (super)derivations. Let $A = \ev
A \oplus \od A$ be an associative superalgebra. Then its
endomorphism algebra $\End_K(A)$ is naturally $\Z_2$-graded with
\[
\End_K(A)_i = \{f \in \End_K(A) \vert\; f(A_j)\subseteq A_{j+i},
\text{ for }j\in \Z_2\}, \quad i\in \Z_2.
\]
Let $\Der_i(A)$, $i \in \Z_2$, be the subspace of all $\delta \in
\End_K(A)_i$ such that
\[
\delta(xy) = (\delta x)y + (-1)^{i|x|}x (\delta y)
\]
for all homogeneous $x, y \in A$.

The Lie superalgebra of derivations of $A$
\[
\Der(A) = \ev \Der (A) \oplus \od \Der (A)
\]
is a restricted Lie subalgebra of $\End_K(A)$.

\subsection{}\label{sec:KWproperty}
Let $\g$ be a restricted Lie superalgebra. For $\chi\in \ev \g^*$,
we always regard $\chi \in \g^*$ by setting $\chi(\od \g) = 0$.
Denote the centralizer of $\chi$ in $\g$ by $\g_{\chi} =
\g_{\chi,\bar 0} + \g_{\chi,\bar 1}$, where $\g_{\chi,i} = \{ y \in
\g_i \vert \; \chi([y, \g]) = 0 \}$ for $i \in \Z_2$. Set $d_0
|\,d_1 = \scodim \g_\chi$. It is well-known that $d_0$ is even
whereas $d_1$ could be odd.

We recall here the following superalgebra generalization of the
Kac-Weisfeiler Conjecture, which is formulated in \cite{WZ1}.

\vspace{.2cm}

{\noindent \bf Super KW Conjecture}. {\em The dimension of every
$\rea{\chi}{\g}$-module is divisible by $p^{\frac{d_0}{2}}2^{\lfloor
\frac{d_1}{2} \rfloor}$.}

\subsection{}\label{subsec:bcls} The basic classical Lie superalgebras over the complex field
$\mathbb{C}$ were classified independently by Kac \cite{Kac}, and
Scheunert-Nahm-Rittenberg \cite{SNR}. Those Lie superalgebras by
definition admit an even nondegenerate supersymmetric bilinear form,
and the even subalgebras are reductive.

We observe that the basic classical Lie superalgebras are defined
over fields of positive characteristics as well under mild
assumption on $p$ (see \cite[Sect.~2]{WZ1}). The restriction on the
characteristic of fields of definition is listed in the following
table (the general linear Lie superalgebra, though not simple, is
also included).

\vspace{.4cm}

\begin{center}
\begin{tabular}{|c|c|}
\hline
Lie superalgebra & Characteristic of $K$\\
\hline $\gl(m|n)$ & $p > 2$\\
\hline $\mathfrak{sl}(m|n)$ & $p > 2, p\nmid (m-n)$\\
\hline $B(m,n), C(n), D(m,n)$ & $p > 2$ \\
\hline $D(2,1; \alpha)$ & $p > 3$ \\
\hline $F(4)$ & $p >2$ \\
\hline $G(3)$ & $p >3$ \\
\hline
\end{tabular}

\vspace{.1cm}
 TABLE: basic classical Lie $K$-superalgebras
\end{center}

\vspace{.4cm}

Note that for each basic classical Lie superalgebra $\g$, the
restriction on the prime $p$ above makes $p$ automatically {\em
good} for the even subalgebra $\ev \g$ (cf.
\cite[Section~2.6]{Jan2}).


\subsection{} In the following two subsections, we
introduce, following \cite{PS}, a family of associative
superalgebras deformed from the reduced enveloping superalgebras.
This part can be viewed as a super counterpart of
\cite[Sect.~2]{PS}; since the proofs of the statements are
essentially the same as those of the corresponding ones in {\em loc.
cit.}, we will omit them and refer the reader to the original paper.

Let $\g$ be a $n_0\vert\,n_1$-dimensional restricted Lie
superalgebra. A {\em $\g$-superalgebra} is a pair consisting of a
$K$-superalgebra $A$ and a homomorphism $\g \ra \Der A$ of
restricted Lie superalgebras.

Given a linear form $\xi \in \ev \g^*$ and a scalar $\lambda \in K$,
denote by $\raa{\xi}{\lambda}{\g}$ the quotient superalgebra of the
tensor superalgebra $T(\g)$ on the superspace $\g$ by its ideal
$I_{\xi, \lambda}$ generated by all elements $x\otimes y -
(-1)^{|x||y|}y\otimes x -\lambda[x,y]$ for all homogeneous $x, y\in
\g$ and elements $x^{\otimes p}-\lambda^{p-1}\pth{x} -\xi(x)^p\cdot
1$ for all $x \in \ev \g$ . Each $\raa{\xi}{\lambda}{\g}$ is a
$\g$-superalgebra.


If $\lambda =1$, the superalgebra $\raa{\xi}{\lambda}{\g}$ is the
reduced enveloping superalgebra $\rea{\xi}{\g}$; while if $\lambda =
0$, the superalgebra is called the {\em reduced symmetric
superalgebra}, denoted by $\rsa{\xi}{\g}$. Since $x^p -
\xi(x)^p=(x-\xi(x))^p$ for $x\in \ev \g$, by changing of variables
we see that $\rsa{\xi}{\g}$ is isomorphic to the truncated
polynomial superalgebra
$$
K[x_1, \ldots, x_{n_0};y_1,\ldots, y_{n_1}]/(x_1^p,\ldots,
x^p_{n_0};y_1^2,\ldots, y^2_{n_1}),$$ where $K[x_1, \ldots,
x_{n_0};y_1,\ldots, y_{n_1}]$ is the (free) commutative superalgebra
on even generators $\{ x_1, \ldots, x_{n_0}\} $ and odd generators
$\{y_1, \ldots, y_{n_1}\}$. The unique maximal ideal of
$\rsa{\xi}{\g}$ is generated by all $x-\xi(x)\cdot 1$ for $x\in \ev
\g$ and all $y \in \od \g$.

If $t\in K^\times=K\setminus \{0\}$, the map $x \mapsto t^{-1}x$,
where $x\in \g$, extends uniquely to the superalgebra isomorphism
\[
\theta_t: \raa{\xi}{\lambda}{\g} \ra U_{t\xi,t\lambda}(\g).
\]
In particular, if $\lambda \neq 0$, then $\raa{\xi}{\lambda}{\g}
\cong \rea{\lambda^{-1}\xi}{\g}$ as superalgebras. All superalgebra
isomorphism $\theta_t$ are $\g$-equivariant.


\subsection{} A vector bundle $A \ra Z$ over an algebraic variety $Z$ together
with a pair of morphism $\mu : A \times_Z A \ra A$ and $\rho : \g
\times A \ra A$ of algebraic varieties over $Z$ is called a {\em
continuous family of (finite-dimensional) $\g$-superalgebras
parametrized by $Z$} if, for the fiber $A_{\zeta}$ over any point
$\zeta \in Z$,
\begin{itemize}
\item[(1)] the restriction of $\mu$ to $A_{\zeta}\times A_{\zeta}$
gives $A_{\zeta}$ a structure of a finite-dimensional associative
superalgebra.
\item[(2)] the restriction of $\rho$ to $\g \times A_{\zeta}$
induces a homomorphism of restricted Lie superalgebras $\g \ra \Der
A_{\zeta}$.
\end{itemize}

The algebraic variety $Z$ is called the parameter space of the
family. By definition, all $\g$-superalgebras in a family have the
same finite dimension.

The isomorphisms $\theta_t$ allow us pass to a continuous family of
superalgebras parametrized by the projective space $\mathbb{P}(\ev
\g^*\oplus K)$ corresponding to the linear space $\ev\g^* \oplus K$.
Write $(\xi: \lambda)$ for the point of $\mathbb{P}(\ev \g^*\oplus
K)$ represented by the pair $(\xi,\lambda) \neq (0,0)$, where $\xi
\in \ev \g^*$, $\lambda \in K$. Identify $\mathbb{P}(\ev \g^*)$ with
the Zariski closed subset of $\mathbb{P}(\ev \g^*\oplus K)$
consisting of all points $(\xi: \lambda)$ with $\lambda = 0$.
Identify each $\xi \in \ev \g^*$ with the point $(\xi:1)\in
\mathbb{P}(\ev \g^*\oplus K)$.
\begin{proposition}
The set of superalgebras $\raa{\xi}{\lambda}{\g}$ with
$(\xi:\lambda) \in \mathbb{P}(\ev \g^*\oplus K)$ is a continuous
family of $\g$-superalgebras parametrized by $\mathbb{P}(\ev
\g^*\oplus K)$ such that the superalgebras corresponding to the
points $\xi \in \ev\g^*$ and $(\xi:0)\in \mathbb{P}(\ev \g^*)$ of
the parameter space are $\g$-equivariantly isomorphic to
$\rea{\xi}{\g}$ and $\rsa{\xi}{\g}$, respectively.
\end{proposition}
\begin{proof}
The proof is the same as that of \cite[Proposition~2.2]{PS}, and
will be omitted here.
\end{proof}

\begin{lemma}\label{lem:dim-closed}
Let $\pi : A \ra Z$ be a continuous family of $\g$-superalgebras
parametrized by an algebraic variety $Z$. Then, for any positive
integer $d$, the set of all points $\zeta \in Z$ such that the
corresponding superalgebra $A_{\zeta}$ contains a $\g$-invariant
two-sided ideal of dimension $d$ is closed in $Z$.
\end{lemma}
\begin{proof}
For a superspace $V$, let $\sigma: V \ra V$ be the linear
transformation whose action on the homogeneous elements is given by
\[
\sigma(v) = (-1)^{|v|}v.
\]
Then a subspace $W$ of $V$ is graded if and only if $\sigma(W)=(W)$.

Let $\varphi: G_d(A)\ra Z$ be the Grassmann bundle of
$d$-dimensional subspaces corresponding to the vector bundle $\pi:
A\ra Z$. Then the subvariety $G^{\text{gr}}_d(A) \subseteq G_d(A)$
of graded subspaces of dimension $d$ is closed.

Given this, the rest of the proof is the same as the proof of
\cite[Lemma~2.3]{PS}.
\end{proof}

%

\subsection{} In the rest of this section, we study the properties of invariant ideals of the
reduced symmetric superalgebras. This can be viewed as the super
counterpart of \cite[Section~3]{PS}. It turns out that most
statements and their proofs in {\em loc. cit.} generalize to the
super setup trivially. As we did in the previous two subsections, we
will only state the facts without proof when their proofs are
straightforward generalization of the corresponding ones in {\em
loc. cit.}.

Let $\p$ be a restricted subalgebra of $\g$, and $\xi \in \ev \g^*$.
For any $\rea{\xi}{\p}$-module $V$, the superspace
\[
\widetilde{V}= \Hom_{\rea{\xi}{\p}}(\rea{\xi}{\g}, V)
\]
carries a standard $\rea{\xi}{\g}$-module structure given by
\[
(xf)(v)=(-1)^{|x|(|f|+|v|)}f(vx)
\]
where $x,v \in \rea{\xi}{\g}$, and $f \in
\Hom_{\rea{\xi}{\p}}(\rea{\xi}{\g}, V)$ are homogeneous elements.
This module is called the $\rea{\xi}{\g}$-module {\em coinduced}
from $V$.

Let $A$ be a $\p$-superalgebra. The restricted $\g$-module
$\widetilde{A} = \Hom_{\rea{0}{\p}}(\rea{0}{\g}, A)$ coinduced from
$A$ carries a superalgebra structure such that the $\g$ acts on
$\widetilde{A}$ as superderivations. The multiplication in
$\widetilde{A}$ is given by the formula
\[
(f\cdot g)(u) = \sum_{(u)} (-1)^{|g||u_{(1)}|}f(u_{(1)})g(u_{(2)}),
\]
where $f,g \in \widetilde{A}$ and $u \in \rea{0}{\g}$ are
homogenous, and where $u \mapsto \sum u_{(1)}\otimes u_{(2)}$ is the
comultiplication of $\rea{0}{\g}$.

\subsection{}

Let $\p$ be a restricted subalgebra of $\g$. Write
\[
\Fp =\Hom_{\rea{0}{\p}}(\rea{0}{\g}, K),
\]
where $K$ denotes the trivial $\rea{0}{\p}$-module.
\begin{lemma}\label{lem:Fpg}
The superalgebra $\Fp$ is $\g$-simple and commutative. Moreover, it
is isomorphic to a truncated symmetric superalgebra. The unique
maximal ideal $\m(\g, \p)$ of $\Fp$ consists of all $f \in \Fp$
satisfying $f(1)=0$.
\end{lemma}
\begin{proof}
Let $\{x_1, \ldots, x_s\}$ (resp. $\{y_1, \ldots, y_t\}$) be
elements in $\ev \g$ (resp. $\od \g$) such that their images form a
basis for $\ev \g /{\ev \p}$ (resp. $\od \g /{\od \p}$).

Let
\begin{align*}
\ev \La= & \{\mathbf{a}=(a_1, \ldots, a_s)\vert\; 0\leq a_i \leq p-1
\text{ are integers} \};\\
\od \La=& \{\mathbf{b}=(b_1, \ldots, b_r)\vert \; 1 \leq b_1 <
\ldots <b_r \leq t \text{ are integers };0\leq r \leq t\}.
\end{align*}
For $\mathbf{a}=(a_1,\ldots,a_s)$ and
$\mathbf{a}'=(a_1',\ldots,a_s')$ in $\ev \La$, denote $\mathbf{a}!=
\prod(a_i !)$. Write $\mathbf{a}' \leq \mathbf{a}$ if $a_i' \leq
a_i$ for all $i$. Further put ${\mathbf{a} \choose
\mathbf{a}'}=\prod {a_i \choose a_i'}$ when $\mathbf{a'} \leq
\mathbf{a}$. For $\mathbf{b}=(b_1, \ldots, b_r)$ and
$\mathbf{b}'=(b'_1, \ldots, b'_l)$ in $\od \La$, write $\mathbf{b}'
\leq \mathbf{b}$ if $(b'_1, \ldots, b'_l)$ appears in $(b_1, \ldots,
b_r)$ as a subsequence. Also, when $\mathbf{b}' \leq \mathbf{b}$,
define $\text{sgn}(\mathbf{b}',\mathbf{b})$ to be the sign of the
permutation of sequence $\mathbf{b}$ given by $(\mathbf{b}',
\mathbf{b}\setminus \mathbf{b}')$, where $\mathbf{b}\setminus
\mathbf{b}'$ denotes the subsequence of $\mathbf{b}$ formed by
removing the subsequence $\mathbf{b}'$ from $\mathbf{b}$.

For $\mathbf{a}=(a_1,\ldots, a_s) \in \ev\La$ and $\mathbf{b}=(b_1,
\ldots, b_r) \in \od \La$, write
\[
e^{(\mathbf{a},\mathbf{b})}=x_1^{a_1}\cdots x_s^{a_s}y_{b_1}\cdots
y_{b_r}.
\]
Then $\rea{0}{\g}$ is a free $\rea{0}{\p}$-module on basis
\[
\{e^{(\mathbf{a},\mathbf{b})} \vert \; \mathbf{a} \in
\ev\La,\;\mathbf{b}\in \od\La\}.
\]
The comultiplication of $\rea{0}{\g}$ on
$e^{(\mathbf{a},\mathbf{b})}$ is given by
\begin{equation}\label{equ:comultiplication}
\Delta(e^{(\mathbf{a},\mathbf{b})})=\sum_{\mathbf{a'}\leq\mathbf{a};\mathbf{b}'\leq
\mathbf{b}}{\mathbf{a} \choose
\mathbf{a}'}\text{sgn}(\mathbf{b}',\mathbf{b})e^{(\mathbf{a}',\mathbf{b}')}\otimes
e^{(\mathbf{a}-\mathbf{a}',\mathbf{b}\setminus\mathbf{b}')}.
\end{equation}

Let $\phi_i \in \ev \Fp$ (resp. $\psi_j \in \od \Fp$) be the dual
element of $x_i$ for $1 \leq i \leq s$ (resp. $y_j$ for $1 \leq j
\leq t$).

Equation~(\ref{equ:comultiplication}) inductively shows that, for
$\mathbf{a}, \mathbf{a}'\in \ev \La$ and $\mathbf{b}, \mathbf{b}'
\in \od \La$,
\[
\phi^{\mathbf{a}}\psi^{\mathbf{b}}(e^{(\mathbf{a}',\mathbf{b}')})=
\mathbf{a}!
\delta_{(\mathbf{a},\mathbf{b}),(\mathbf{a}',\mathbf{b}')},
\]
where we use the notation
\[
\phi^{\mathbf{a}}=\phi_1^{a_1}\cdots \phi_s^{a_s}, \qquad
\psi^{\mathbf{b}}=\psi_{b_1}\cdots \psi_{b_r},
\]
for $\mathbf{a}=(a_1, \ldots, a_s)$, $\mathbf{b}=(b_1, \ldots,
b_r)$.

Then $\Fp$ is an associative superalgebra with unit element and
generators $\phi_1, \ldots, \phi_s$ and $\psi_1,\ldots, \psi_t$,
which satisfy $\phi^p_i=0$ and $\psi^2_j=0$ for all $i,j$. As its
dimension is $p^s2^t$, there is an isomorphism
\[
K[x_1, \ldots, x_{s};y_1,\ldots, y_{t}]/(x_1^p,\ldots,
x^p_{s};y_1^2,\ldots, y^2_{t}) \cong \Fp.
\]

To see it is $\g$-simple, we note inductively from
equation~(\ref{equ:comultiplication}) that the action of $\g$ on
some of the basis vectors of $\Fp$ is given as follows:
\begin{eqnarray*}
\lefteqn{x_i\cdot \phi_1^{a_1}\cdots \phi_s^{a_s}\psi^{\mathbf{b}}}\\
& =& \begin{cases} 0 & \text{if }a_i =0,\\
\lambda \phi_1^{a_1}\cdots \phi_i^{a_i-1}\cdots
\phi_s^{a_s}\psi^{\mathbf{b}}
& \text{if }2\leq a_i\leq p-1,\\
\mu \phi_1^{a_1}\cdots
\phi_{i-1}^{a_{i-1}}\phi_{i+1}^{a_{i+1}}\cdots
\phi_s^{a_s}\psi^{\mathbf{b}}+ \nu \phi_1^{a_1}\cdots
\phi_i^{p-1}\cdots \phi_s^{a_s}\psi^{\mathbf{b}} & \text{if } a_i=1;
\end{cases}\\
\lefteqn{y_j \cdot \psi_{j_1}\cdots \psi_{j_r}} \\
&=& \begin{cases} 0 & \text{if } j \notin \{j_1,\ldots, j_r\},\\
\pm \psi_{j_1}\cdots\hat{\psi_j}\cdots \psi_{j_r} &
\text{otherwise}.
\end{cases}
\end{eqnarray*}
where $\lambda$, $\mu$, and $\nu$ are in $K$ with $\lambda$, $\mu$
nonzero.

Given any nonzero element in $\Fp$, by applying a suitable sequence
of $x_i$'s and $y_j$'s, we will eventually arrive at a linear
combination of basis vectors $\phi^{\mathbf{a}}\psi^{\mathbf{b}}$
with nonzero constant term. On the other hand, since $\phi_i^p=0$
and $\psi_j^2=0$, every nonzero $\g$-invariant ideal is nilpotent
and contains an element with nonzero constant term. It has to be the
whole thing. Hence $\Fp$ is $\g$-simple. The rest of the statement
is clear.
\end{proof}

Let $B=\ev B \oplus \od B$ be a finite dimensional unital
commutative associative $\g$-superalgebra. The superalgebra $B$ is
said to be $\g$-simple if it contains no nonzero $\g$-invariant
ideals. Arguing as in \cite[3.2]{PS}, we can show that if $B$ is
$\g$-simple, then it is a local superalgebra with the unique maximal
ideal $\m = \ev \m \oplus \od B$, where $\ev \m$ consists of the
elements $b \in \ev B$ such that $b^p=0$.

\begin{proposition}\label{prop:simple g-superalgebra}
Let $B$ be a $\g$-simple finite dimensional unital commutative
$\g$-superalgebra. Denote by $\m$ the maxiaml ideal, and by $\p$ the
normalizer of $\m$ in $\g$. Then there is a canonical
$\g$-equivariant superalgebra isomorphism $B \cong \Fp$.
\end{proposition}
\begin{proof}
The proof is similar to the proof of \cite[Thm.~3.2]{PS}, and will
be skipped here.
\end{proof}

\subsection{} Let $B$ be a commutative
$\g$-superalgebra and $\xi \in \ev \g^*$. By a $(B,
\rea{\xi}{\g})$-module, we mean a $\rea{\xi}{\g}$-module which is
also a module over superalgebra $B$ such  that the module structure
map $B\otimes M \ra M$ is a $\g$-module homomorphism. A $(B,
\g)$-superalgebra is a $K$-superalgebra $C$, which is simultaneously
a $B$-superalgebra and $\g$-superalgebra and a $(B,
\rea{0}{\g})$-module.

Now let $B=\Fp$. For any $\rea{\xi}{\p}$-module $V$, the coinduced
$\rea{\xi}{\g}$-module $\widetilde{V}$ carries a canonical
$(\Fp,\rea{\xi}{\g})$-module structure given by
\[
(f\cdot \psi)(u) = \sum_{(u)}
(-1)^{|\psi||u_{(1)}|}f(u_{(1)})\psi(u_{(2)}),
\]
where $f \in \Fp$, $\psi \in \widetilde{V}$, and $u \in
\rea{\xi}{\g}$ are homogeneous.

If $A$ is a $\p$-superalgebra, then the $(\Fp, \rea{0}{\g})$-module
$\widetilde{A} = \Hom_{\rea{0}{\p}}(\rea{0}{\g}, A)$ has a
$\g$-invariant multiplication, it is $\Fp$-bilinear as well.
Therefore, $\widetilde{A}$ is an $(\Fp, \g)$-superalgebra.

\begin{proposition}\label{prop:imprimitivity}
Let $M$ be an $(\Fp, \rea{\xi}{\g})$-module and $C$ an $(\Fp,
\g)$-superalgebra. Then,
\begin{itemize}
\item[(i)] $M \cong \Hom_{\rea{\xi}{\p}}(\rea{\xi}{\g},
M/{\m(\g,\p)M})$ as $(\Fp, \rea{\xi}{\g})$-module.
\item[(ii)] $C \cong \Hom_{\rea{0}{\p}}(\rea{0}{\g}, C/{\m(\g,\p)C})$
as $(\Fp, {\g})$-superalgebras.
\end{itemize}
\end{proposition}
\begin{proof}
The proof is similar to proof of \cite[Thm~3.3]{PS}, and will be
skipped here.
\end{proof}

\subsection{}
Let $\xi \in \ev \g^*$. Recall the centralizer of $\xi$ in $\g$ is
denoted by $\g_\xi$, which is a restricted Lie subalgebra. Put $d_0
\vert\, d_1 = \scodim \g_{\xi}$.

\begin{proposition}\label{prop:sym}
Let $\xi \in \ev \g^*$ and $d_0 \vert\, d_1 = \scodim\g_{\xi}$. Then
each $\g$-invariant ideal of $\rsa{\xi}{\g}$ has codimension
divisible by $p^{d_0}2^{d_1}$. Among them, there is a unique maximal
one of codimension $p^{d_0}2^{d_1}$.
\end{proposition}
\begin{proof}
The proof, which uses Propositions~\ref{prop:simple g-superalgebra}
and~\ref{prop:imprimitivity}(i), is similar to proof of
\cite[Thm~3.4]{PS}, and will be skipped here.
\end{proof}

%
\begin{remark}\label{rem:PS}
Let $\mathscr L$ be an $n$-dimensional restricted Lie algebra, and
let $r$ be the minimal dimension of the centralizers of all $\chi
\in \mathscr{L}^*$. It is conjectured by Kac-Weisfeiler that the
maximal dimension $M(\mathscr{L})$ of simple $\mathscr L$-modules is
$p^{\frac{n-r}{2}}$. Following \cite{PS}, we refer to this
conjecture as KW1 conjecture, which is still open.

In \cite{PS}, Premet and Skryabin showed that
\begin{itemize}
\item[(\dag)] the set of $\chi \in \mathscr{L}^*$ such that
$\rea{\chi}{\mathscr{L}}$ has all its simple modules having the
maximal dimension $M(\mathscr{L})$ is nonempty and Zariski open in
$\mathscr{L}^*$.
\end{itemize}
Using deformation arguments, they then showed that
\begin{itemize}
\item[(\ddag)] if there is $\chi \in \mathscr{L}^*$ whose centralizer is a toral
subalgebra of $\mathscr L$, then there is a nonempty and Zariski
open subset $W$ of $\mathscr{L}^*$ such that $\xi \in W$ implies
that all simple $\rea{\xi}{\mathscr{L}}$-modules have dimension
$p^{\frac{n-r}{2}}$.
\end{itemize}
This, together with $(\dag)$, confirms KW1 conjecture for such
$\mathscr{L}$.

Along the line in this section, we can establish the corresponding
statement of $(\ddag)$ in the superalgebra setting. However, it is
not clear how to generalize $(\dag)$ to a general restricted Lie
superalgebra. This is mainly due to the fact that the universal
enveloping superalgebra is in general not a prime ring (see \cite{B}
for a counterexample over the complex numbers), which is crucial in
the proof of $(\dag)$ in \cite{PS}.
\end{remark}

\section{Proof of super KW property for basic classical Lie
superalgebras}\label{sec:SKW}
\subsection{} In this subsection we first recall some basic facts on the method of
associated cones, following \cite[Sect.~5.1]{PS}.

Let $V$ be a finite dimensional vector space over $K$. For an ideal
$I$ of the symmetric algebra $S(V^*)$, let $\text{gr}I$ denote the
homogeneous ideal of $S(V^*)$ with the property that $g \in
\text{gr}I \cap S^r(V^*)$ if and only if there is $\tilde{g} \in I$
such that
\[
\tilde{g} - g \in \oplus_{j<r}S^j(V^*).
\]
Identify $S(V^*)$ with the algebra of polynomial functions on $V$.
Given a subset $X \subseteq V$, let
\[
I_X = \{ g \in S(V^*) \vert \; g(X) = 0\}
\]
be the ideal associated to it. The set
\[
\mathbb{K}X :=\{ v\in V \vert \; f(v)=0 \text{ for all } f\in
\text{gr}I_X\}.
\]
is called the {\em cone associated with} $X$. It is a Zariski closed
conical subset of $V$. We identify $V$ (resp. $\mathbb{P}(V)$) with
the subset of $\mathbb{P}(V\oplus K)$ consisting of all points
$(v:1)$ (resp., $(v:0)$) with $v \in V$ (resp. $v\in
V\setminus\{0\}$). Let $\overline{X}^P$ (resp. $\overline{X}$)
denote the Zariski closure of $X$ in $\mathbb{P}(V\oplus K)$ (resp.,
in $V$). The following facts are easy to prove
\begin{equation}\label{equ:closures}
\overline{X}^P \cap V = \overline{X} \text{  and  }\overline{X}^P
\cap \mathbb{P}(V) =\mathbb{P}(\mathbb{K}X),
\end{equation}
where $\mathbb{P}(\mathbb{K}X) \subseteq \mathbb{P}(V)$ denotes the
projectivisation of the conical subset $\mathbb{K}X$.

\subsection{} Now let $\g$ be a restricted Lie superalgebra. For a pair of nonnegative integers
$(d_0|\, d_1)$ with $d_0$ even, let $\mathscr{X}_{d_0, d_1}$ denote
the set of all $\xi \in \ev \g^*$ such that the algebra
$\rea{\xi}{\g}$ has a module of finite dimension not divisible by
$p^{\frac{d_0}{2}}2^{\lfloor \frac{d_1}{2}\rfloor}$, and let
$\mathscr{X}'_{d_0,d_1} \subseteq \mathbb{P}(\ev \g^* \oplus K)$ be
the subset of all points $(\xi: \lambda)$ satisfying
$\raa{\xi}{\lambda}{\g}$ has a $\g$-invariant ideal of codimension
not divisible by $p^{d_0}2^{d_1}$. Set
\[
\mathscr{Y}_{d_0, d_1} = \{ \xi \in \ev \g^* \vert \; \codim_{\ev
\g}\g_{\xi,\bar{0}}< d_0 \text{ or }\codim_{\od \g}\g_{\xi,\bar{1}}<
d_1 \}.
\]
Note $\mathscr{X}_{d_0,2k+1} = \mathscr{X}_{d_0,2k+2}$, but this is
not the case for $\mathscr{X}'_{d_0,d_1}$ and $\mathscr{Y}_{d_0,
d_1}$.

By Lemma~\ref{lem:dim-closed}, $\mathscr{X}'_{d_0,d_1}$ is closed.
The set $\mathscr{Y}_{d_0,d_1}$ is obviously conical, and let
$\mathbb{P}(\mathscr{Y}_{d_0,d_1}) \subseteq \mathbb{P}(\ev \g^*)$
be its projectivization. By Proposition~\ref{prop:sym}, $\eta \in
\ev \g^*$ lies in $\mathscr{Y}_{d_0,d_1}$ if and only if
$\rsa{\eta}{\g}$ has a $\g$-invariant ideal with codimension not
divisible by $p^{d_0}2^{d_1}$. Therefore,
\begin{equation}\label{equ:Y}
\mathscr{X}'_{d_0,d_1} \cap \mathbb{P}(\ev
\g^*)=\mathbb{P}(\mathscr{Y}_{d_0,d_1}).
\end{equation}
Hence $\mathbb{P}(\mathscr{Y}_{d_0,d_1})$ is closed in
$\mathbb{P}(\ev \g^*)$, and so $\mathscr{Y}_{d_0,d_1}$ is Zariski
closed in $\ev \g^*$.

\begin{proposition}\label{prop:XY}
We have $\mathbb{K}\mathscr{X}_{d_0, d_1} \subseteq
\mathscr{Y}_{d_0, d_1}$ for any pair of nonnegative integers $(d_0|\,d_1)$ with $d_0$ even. 
\end{proposition}

\begin{proof}
We claim that $\mathscr{X}_{d_0,d_1}\subseteq \mathscr{X}'_{d_0,d_1}
\cap \ev \g^*$. Indeed, suppose $\xi \in \ev\g^* \setminus
(\mathscr{X}'_{d_0,d_1} \cap \ev \g^*)$. Then each two-sided ideal
of $\rea{\xi}{\g}$ is of codimension divisible by $p^{d_0}2^{d_1}$
since all the two-sided ideals of $\rea{\xi}{\g}$ are
$\g$-invariant. Let $V$ be a simple module of $\rea{\xi}{\g}$ and
let $J=\text{Ann}_{\rea{\xi}{\g}}V$ be its annihilator in
$\rea{\xi}{\g}$. Then by \cite[Section~2]{LBF}$, \rea{\xi}{\g}/J$ is
a simple superalgebra over $K$ with the unique simple module $V$
since $K$ is algebraically closed. Then either (1) $\rea{\xi}{\g}/J$
is of type $M$ with dimension $a^2$ for some natural number $a$; or
it is of type $Q$ with dimension $2b^2$ for some natural number $b$.
In case (1), the dimension of $V$ is $a$. Then since
$p^{d_0}2^{d_1}$ divides $a^2=\dim\rea{\xi}{\g}/J$, we will have
$p^{\frac{d_0}{2}}2^{\lfloor \frac{d_1}{2}\rfloor}$ divides $a$. In
case (2), the dimension of $V$ is $2b$. Since $p^{d_0}2^{d_1}$
divides $2b^2=\dim\rea{\xi}{\g}/J$, $p^{\frac{d_0}{2}}2^{\lceil
\frac{d_1}{2}\rceil}$ divides $b$. In either case, the
$\rea{\xi}{\g}$-module $V$ has dimension divisible by
$p^{\frac{d_0}{2}}2^{\lfloor \frac{d_1}{2}\rfloor}$, which implies
$\xi \notin \mathscr{X}_{d_0,d_1}$. The claim is proved.

The claim implies that $\overline{\mathscr{X}}_{d_0,d_1}^P \subseteq
\mathscr{X}'_{d_0,d_1}$, since $\mathscr{X}'_{d_0,d_1}$ is closed by
Lemma~\ref{lem:dim-closed}. Then we have
$\overline{\mathscr{X}}_{d_0,d_1}^P \cap \mathbb{P}(\ev \g^*)
\subseteq \mathscr{X}'_{d_0,d_1}\cap \mathbb{P}(\ev \g^*)$, this
means $\mathbb{P}(\mathbb{K}\mathscr{X}_{d_0,d_1}) \subseteq
\mathbb{P}(\mathscr{Y}_{d_0,d_1})$ by (\ref{equ:closures}) and
(\ref{equ:Y}). But since both $\mathbb{K}\mathscr{X}_{d_0,d_1}$ and
$\mathscr{Y}_{d_0,d_1}$ are conical, we deduce that
$\mathbb{K}\mathscr{X}_{d_0,d_1} \subseteq \mathscr{Y}_{d_0,d_1}$,
as desired.
\end{proof}

\subsection{} Let $G(\ev \g)$ denote the group of all automorphisms
of $\ev \g$ preserving the $[p]$th power map, i.e., automorphisms
$g$ satisfying $g(\pth x)=\pth{g(x)}$ for all $x \in \ev \g$. Let
$\Omega(\eta)$ denote the $G(\ev \g)$-orbit of $\eta \in \ev \g^*$.

For $\chi \in \ev \g^*$, define
\begin{align*}
l_0(\chi) = \min_{\xi \in \mathbb{K}\Omega(\chi)}\dim \g_{\xi,\bar{0}},\\
l_1(\chi) = \min_{\xi \in \mathbb{K}\Omega(\chi)}\dim
\g_{\xi,\bar{1}}.
\end{align*}

\begin{theorem}\label{thm-div}
Let $\g$ be an $(n_0|\,n_1)$-dimensional restricted Lie
superalgebra, and $\chi \in \ev \g^*$. Write $l_i=l_i(\chi)$ for $i
\in \Z_2$, and $d_0|\,d_1=\scodim_{\g}\g_{\chi}$. Then,
\begin{itemize}
\item[(i)] Each finite dimensional $\rea{\chi}{\g}$-module has
dimension divisible by
$p^{\frac{n_0-l_0}{2}}2^{\lfloor\frac{n_1-l_1}{2}\rfloor}$.
%

\item[(ii)] If all nonzero scalar multiples of $\chi$ are $G(\ev
\g)$-conjugate, then the dimensions of all finite dimensional
$\rea{\chi}{\g}$-modules are divisible by
$p^{\frac{d_0}{2}}2^{\lfloor \frac{d_1}{2}\rfloor}$, i.e. the super
KW conjecture holds for the algebra $\rea{\chi}{\g}$.
\end{itemize}
\end{theorem}

\begin{proof}
To prove part (i), we treat the p- and 2-divisibility separately.
Suppose that $\rea{\xi}{\g}$ has a finite dimensional module $V$
such that $\dim V$ is not divisible by $2^{\lfloor
\frac{n_1-l_1}{2}\rfloor}$. Then $\chi \in \X{0}{n_1-l_1}$ by the
definition of $\X{0}{n_1-l_1}$. (Note that in addition, $\chi \in
\X{0}{n_1-l_1+1}$ when $n_1-l_1$ is odd, while $\chi \in
\X{0}{n_1-l_1-1}$ when $n_1-l_1$ is even. But we do not need this.)
Since for any $g \in G(\ev \g)$, the algebras $\rea{g(\chi)}{\g}$
and $\rea{\chi}{\g}$ are isomorphic. It follows that $\Omega(\chi)
\subseteq \X{0}{n_1-l_1}$. But then $\mathbb{K}\Omega(\chi)
\subseteq \mathbb{K}\X{0}{n_1-l_1}$. As $\mathbb{K}\X{0}{n_1-l_1}
\subseteq \Y{0}{n_1-l_1}$ by Proposition~\ref{prop:XY}, we have
\[
\text{codim}_{\od \g}\g_{\xi, \bar{1}} < n_1-l_1
\]
for any $\xi \in \mathbb{K}\Omega(\chi)$, which contradicts the
choice of $l_1$.

The $p$-divisibility can be proved similarly.

For part (ii), note first that $(\chi:0) \in
\overline{K^\times\chi}^P$. Since by assumption that $K^\times\chi$
is contained in a single $G(\ev\g^*)$-orbit, we have, by
equation~(\ref{equ:closures}),
\[
(\chi:0) \in \overline{K^\times\chi}^P \cap
\mathbb{P}(\ev\g^*)\subseteq \overline{\Omega(\chi)}^P \cap
\mathbb{P}(\ev\g^*)=\mathbb{P}(\mathbb{K}\Omega(\chi)).
\]
Thus $\chi \in \mathbb{K}\Omega(\chi)$, and as a result, $l_i \leq
n_i-d_i$ for $i \in \Z_2$. From here, (ii) follows from (i).
\end{proof}

\subsection{} Now let $\g$ be one of the basic
classical Lie superalgebras as in Section~\ref{subsec:bcls}. Recall
that the even subalgebra $\ev \g$ is the Lie algebra of a reductive
group $G_{\bar{0}}$, and that $\g$ admits an even nondegenerate
$G_{\bar{0}}$-invariant bilinear form. Given the bilinear form, we
can speak of nilpotent $p$-characters, i.e. those which correspond
to nilpotent elements in $\ev \g$ under the isomorphism $\ev \g
\cong \ev \g^*$ induced by the bilinear form.

We are now ready to give an alternative proof of
\cite[Theorem.~4.3]{WZ1}.

\begin{theorem}\label{thm:SKW-nilpotent}
Let $\g$ be as in Section~\ref{subsec:bcls}, and let $\chi \in \ev
\g^*$ be nilpotent. Write $d_0 \vert\, d_1 = \scodim\g_{\chi}$. Then
the dimension of every finite dimensional $\rea{\chi}{\g}$-module
$V$ is divisible by $p^{\frac{d_0}{2}}2^{\lfloor
\frac{d_1}{2}\rfloor}$.
\end{theorem}
\begin{proof}
By \cite[Theorem~2.8.1]{Jan2}, $\ev G$ has finitely many orbits in
$\ev\g$. Thus $\ev G$ has finitely many coadjoint orbits in
$\ev\g^*$ via the $\ev G$-equivariant isomorphism $\ev \g\cong
\ev\g^*$. If $\chi \in \ev \g^*$ is nilpotent, so is $K^\times\chi$.
Then by \cite[Lemma~2.10]{Jan2}, $K^\times\chi$ is contained in the
$G_{\bar{0}}$-orbit of $\chi$. Now since $\Ad G_{\bar{0}} \subseteq
G(\ev \g)$, we have $K^\times\chi \subseteq G_{\bar{0}}\cdot\chi
\subseteq \Omega(\chi)$. Hence by Theorem~\ref{thm-div} (ii), the
dimension of every finite dimensional $\rea{\chi}{\g}$-module $V$ is
divisible by $p^{\frac{d_0}{2}}2^{\lfloor \frac{d_1}{2}\rfloor}$,
i.e., the super KW conjecture holds for $\rea{\chi}{\g}$.
\end{proof}

\begin{remark}
In a similar fashion as in Theorem~\ref{thm:SKW-nilpotent}, we can
use Theorem~\ref{thm-div} to give an alternative proof of super KW
conjecture for the queer Lie superalgebra with nilpotent
$p$-characters (\cite[Theorem~4.4]{WZ2}).
\end{remark}

%
%

Now together with \cite[Remarks~2.5 and 4.6, Theorem~5.2]{WZ1}, we
have strengthened the super KW property for basic classical Lie
superalgebras as follows. We remark here that
\cite[Theorem~5.2]{WZ1} remains valid for basic classical Lie
superalgebras with assumption on $p$ as in
Section~\ref{subsec:bcls}.

\begin{theorem}[Super Kac-Weisfeiler Conjecture] \label{th:KW}
Let $\g$ be a basic classical Lie superalgebra as in
Section~\ref{subsec:bcls}, and let $\chi \in \ev \g^*$. Let
$d_0|\,d_1=\scodim \g_\chi$. Then the dimension of every
$\rea{\chi}{\g}$-module $M$ is divisible by
$p^{\frac{d_0}{2}}2^{\lfloor \frac{d_1}{2} \rfloor}$.
\end{theorem}

\section{Semisimple $p$-characters for basic classical Lie
superalgebras}\label{sec:ss} Now we turn our attention to basic
classical Lie superalgebras $\g$ (Sect.~\ref{subsec:bcls}) with a
semisimple $p$-character $\chi \in \ev\g^*$ (see below for a
definition). Our purpose is to give a semisimplicity criterion for
the reduced enveloping superalgebra $\rea{\chi}{\g}$.

Let $\g$ be one of the basic classical Lie superalgebras as in
Sect.~\ref{subsec:bcls}. Fix a Cartan subalgebra $\h$ of $\g$ (and
of $\ev \g$). It defines the set of roots $\Delta= \ev \Delta \cup
\od \Delta$, where $\ev \Delta$ (resp. $\od \Delta$) is the set of
even (resp. odd) roots. Let $W$ be the Weyl group of $\ev\g$. The
$\ev G$-invariant bilinear form on $\g$ induces a $W$-invariant
bilinear form $(.,.)$ on $\h^*$. Put
\begin{align*}
\ev{\overline \Delta} & =\{\alpha \in \ev \Delta |\; \frac{1}{2}
\alpha \notin \od \Delta\};\\
\od{\overline \Delta} & =\{\alpha \in \od \Delta |\; 2 \alpha \notin
\ev \Delta\}=\{\alpha \in \od\Delta|\; (\alpha,\alpha)=0\}.
\end{align*}
For $\alpha \in \Delta$, let $H_\alpha$ and $X_\alpha$ be a choice
of coroot and root vectors respectively.

Let $\chi \in \ev \g^*$ be a $p$-character satisfying
$\chi(X_\alpha)=0$ for all $\alpha \in \ev \Delta$. A $p$-character
which is $\ev G$-conjugate to one of such $\chi$ is called {\em semisimple}. 

\subsection{} Fix an arbitrary set of
simple roots $\Pi$ of $\Delta$. It determines a set of positive
roots ${}^\Pi\Delta^+$. Denote by ${}^\Pi\ev{\Delta}^+$,
${}^\Pi\od{\Delta}^+$, ${}^\Pi\ev{\overline{\Delta}}^+$, and
${}^\Pi\od{\overline{\Delta}}^+$ the subsets of positive roots in
the sets $\ev \Delta$, $\od \Delta$, etc. respectively. Let
\[
\g = {}^\Pi\n^- \oplus \h \oplus {}^\Pi\n^+
\]
be the corresponding triangular decomposition. Put
${}^\Pi\mathfrak{b}= \h \oplus {}^\Pi\n^+$. Let
${}^\Pi\rho={}^\Pi\ev \rho - {}^\Pi\od \rho$, where ${}^\Pi\ev \rho$
(resp. ${}^\Pi\od \rho$) is the half sum of positive even (resp.
odd) roots.

For $\la \in \La_\chi:=\{ \la \in \h^* |\; \la(h)^p-\la(\pth
h)=\chi(h)^p \text{ for all } h \in \h\}$, the baby Verma module
$Z^\Pi_\chi(\la)$ is defined to be
\[
Z^\Pi_\chi(\la) := \rea{\chi}{\g}
\otimes_{\rea{\chi}{{}^\Pi\mathfrak{b}}} K_\la,
\]
where $K_\la$ is the one-dimensional
$\rea{\chi}{{}^\Pi\mathfrak{b}}$-module upon which $\h$ acts via
multiplication by $\la$ and ${}^\Pi\n^+$ acts as zero. Write $v_\la
= 1 \otimes 1_\la$ in $Z^\Pi_\chi(\la)$.

Index roots in ${}^\Pi\Delta^+$ by $\{1, 2, \ldots, N=|\Delta|/2\}$
in a way that is compatible with heights of roots. Here by
compatible we mean that the shorter the root is in height the
smaller it is indexed. For $\alpha \in {}^\Pi\Delta^+$, put
\[
m_\alpha =\begin{cases} p-1 & \text{if } \alpha \in
{}^\Pi\ev{\Delta}^+;\\
1 & \text{if } \alpha \in {}^\Pi\od{\Delta}^+. \end{cases}
\]
\begin{lemma}\label{lem:lowest vector}
Any nonzero submodule $S$ of $Z^\Pi_\chi(\la)$ contains vector
$X_{-\alpha_1}^{m_{\alpha_1}} \cdots X_{-\alpha_N}^{m_{\alpha_N}}
v_\la$.
\end{lemma}
\begin{proof}
The proof is similar to the proof of \cite[Proposition~4]{Rud} and
will be skipped here.
\end{proof}

\begin{lemma}\label{lem:top poly-existence}
In $U(\g)$, we have
\[
X_{\alpha_1}^{m_{\alpha_1}}\cdots
X_{\alpha_N}^{m_{\alpha_N}}X_{-\alpha_1}^{m_{\alpha_1}} \cdots
X_{-\alpha_N}^{m_{\alpha_N}} - {}^\Pi\Phi \in U(\g){}^\Pi \n^+,
\]
where ${}^\Pi\Phi$ is a polynomial in $\{H_\alpha\vert\;\alpha \in
\Pi\}$ of degree $(\frac{(p-1)|\ev \Delta|}{2}+ \frac{|\od
\Delta|}{2})$.
\end{lemma}
\begin{proof}
The proof is similar to the proof of \cite[Proposition~5]{Rud} and
will be omitted here.
\end{proof}

\begin{proposition}\label{prop:irred-verma}
The baby Verma module $Z^\Pi_\chi(\la)$ is irreducible if and only
if ${}^\Pi\Phi(\la) \neq 0$ for $\la \in \La_\chi$.
\end{proposition}
\begin{proof}
Follows readily from Lemmas~\ref{lem:lowest vector} and \ref{lem:top
poly-existence}.
\end{proof}

Finally in this subsection, put ${}^\Pi\Phi'(\la)= {}^\Pi\Phi(\la -
{}^\Pi\rho)$ for $\la \in \h^*$.

\subsection{}
Retain the notations from previous subsection. Let $\delta \in \Pi$
be a simple root. Then it is one of the following three types:
\begin{itemize}
\item[(i)] $\delta \in \ev{\overline{\Delta}}$;
\item[(ii)] $\delta \in \od{\overline{\Delta}}$;
\item[(iii)] $\delta \in \Delta_{\bar{1}} \setminus
\od{\overline{\Delta}}$ with $2\delta \in \Delta_{\bar{0}} \setminus
\ev{\overline{\Delta}}$.
\end{itemize}
For such a $\delta$, we shall denote
\begin{equation*}
\delta^*=\begin{cases} \delta, & \text{in case (i) and (ii);}\\
\{\delta, 2\delta\}, & \text{in case (iii)}.
\end{cases}
\end{equation*}

Let $r_\delta$ be the (even or odd) reflection associated to
$\delta$. When $\delta$ is of type (i), $r_\delta$ is just the even
reflection in $\h^*$ defined by
\begin{equation}\label{equ:reflection}
r_\delta(\la)=\la- \frac{2(\delta,\la)}{(\delta,\delta)}\delta,
\quad \text{ for }\la \in \h^*.
\end{equation}
When $\delta$ is of type (iii), $r_\delta$ is by definition the even
reflection $r_{2\delta}$, which is also given by
formula~(\ref{equ:reflection}). When $\delta$ is of type (ii),
$r_\delta$ is given by the following
\[
r_\delta(\beta)=\begin{cases} -\delta, & \text{if }\beta=\delta;\\
\beta+\delta,&\text{if }(\delta,\beta) \neq 0;\\
\beta, &\text{if }\beta\neq \delta\text{ and }(\delta,\beta)=0.
\end{cases}
\]

It is known that (see, for example, \cite{Ser}) $r_\delta \Pi$ is
the set of simple roots of the positive system
${}^{r_\delta\Pi}\Delta^+:=r_\delta({}^\Pi \Delta^+)$, $-\delta^*
\in {}^{r_\delta\Pi}\Delta^+$, and ${}^{r_\delta\Pi}\Delta^+ \cap
{}^\Pi \Delta^+={}^\Pi \Delta^+\setminus \delta^*$.

By going through the same argument in the previous subsection, we
know that there is a polynomial ${}^{r_\delta \Pi}\Phi$ on $\h^*$ of
degree $(\frac{(p-1)|\ev \Delta|}{2}+ \frac{|\od \Delta|}{2})$
satisfying that ${}^{r_\delta \Pi}\Phi(\la)\neq 0$ if and only if
the baby Verma module $Z_\chi^{r_\delta\Pi}(\la)$ associated to the
positive system ${}^{r_\delta \Pi}\Delta^+$ is irreducible for any
$\la \in \La_\chi$.

For two polynomials $f_1$ and $f_2$, write $f_1 \sim f_2$ if $f_1 =
cf_2$ for some $c \in K^\times$.
\begin{lemma}\label{lem:poly'sim}
We have ${}^{r_\delta \Pi} \Phi' \sim {}^\Pi \Phi'$ for a simple
root $\delta \in \Pi$.
\end{lemma}
\begin{proof}
Let us prove when $\delta$ is of type (iii), the other two cases can
be proved in a similar fashion. First we observe that the vector
$X_{-\delta}X^{p-1}_{-2\delta} v_\la$ in $Z_\chi^\Pi(\la)$ is
annihilated by any root vector $X_\alpha$ for $\alpha \in
{}^{r_\delta \Pi}\Delta^+$. It follows that there is a nontrivial
$U(\g)$-module homomorphism
\[
Z_\chi^{r_\delta \Pi}(\la+\delta) \ra Z_\chi^\Pi(\la).
\]
Since the two baby Verma modules have the same dimension,
$Z_\chi^{r_\delta\Pi}(\la +\delta)$ being reducible will imply that
$Z_\chi^\Pi(\la)$ is reducible. By
Proposition~\ref{prop:irred-verma}, we have
${}^{r_\delta\Pi}\Phi(\la+\delta)$ divides ${}^\Pi\Phi(\la)$, and so
${}^{r_\delta\Pi}\Phi(\la+\delta) \sim {}^\Pi \Phi(\la)$. Hence
${}^{r_\delta\Pi}\Phi'(\la) \sim {}^\Pi \Phi'(\la)$ since
${}^{r_\delta\Pi}\rho={}^\Pi \rho -\delta$.

When $\delta$ is of type (i), then as in the classical case, the
vector $X_{-\delta}^{p-1}v_\la$ in $Z_\chi^\Pi(\la)$ is a singular
vector for the positive system ${}^{r_\delta\Pi}\Delta^+$. We then
can argue the same way as for $\delta$'s of type (iii).

When $\delta$ is of type (ii), we only need to observe that the
vector $X_{-\delta}v_\la$ in $Z_\chi^\Pi(\la)$ is a singular vector
for the positive system ${}^{r_\delta\Pi}\Delta^+$. The rest of the
argument is the same as for $\delta$'s of type (iii).

\end{proof}

Since by applying (even and odd) simple reflections, we are able to
obtain any set $\tilde{\Pi}$ of simple roots from a given set $\Pi$
of simple roots, we conclude by Lemma~\ref{lem:poly'sim} that, the
polynomial ${}^{\tilde{\Pi}}\Phi'$ does not depend on the choice of
$\tilde{\Pi}$ up to ``$\sim$''-equivalence. Thus we can suppress the
left superscript $\tilde{\Pi}$ of ${}^{\tilde{\Pi}}\Phi'$ and write
$\Phi'$ instead.

\begin{proposition}\label{prop:poly'}
We have $\Phi'(\la) \sim \prod_{\alpha \in {}^\Pi
\Delta^+_{\bar{0}}} ((\la|\,\alpha)^{p-1}-1) \cdot \prod_{\beta \in
{}^\Pi \Delta^+_{\bar{1}}}(\la|\,\beta)$, for any set of simple
roots $\Pi$.
\end{proposition}

A different choice of simple roots in Proposition~\ref{prop:poly'}
will only lead to a plus/minus sign in the product on the right hand
side in the Proposition.

\begin{proof}
First observe that if $\od\Delta \setminus
\od{\overline{\Delta}}\neq \emptyset$, then any $\delta \in
\od\Delta \setminus \od{\overline{\Delta}}$ appears as a simple root
in some set $\tilde{\Pi}$ of simple roots. The root vector
$X_\delta$ generates an embedded $\osp(1|\,2)$ in $\g$. Consider the
minimal parabolic subalgebra $\p = \osp(1|\,2)
+{}^{\tilde{\Pi}}\mathfrak{b}$, and the induced module
$Z_\chi^\p(\la)=\rea{\chi}{\p}\otimes_{\rea{\chi}{{}^{\tilde{\Pi}}\mathfrak{b}
}}K_\la$. The $\rea{\chi}{\p}$-module $Z_\chi^\p(\la)$ is merely the
baby Verma module $Z_\chi^{\osp(1|2)}(\la)$ of the embedded
$\osp(1|\,2)$ upon which $\h$ acts as weight multiplication by $\la$
and $X_\alpha$ acts zero for $\alpha \in {}^{\tilde{\Pi}}\Delta^+
\setminus\{\delta, 2\delta\}$. By the transitivity of induced
modules, we have
\[
Z_\chi^{\tilde{\Pi}}(\la) \cong \rea{\chi}{\g}
\otimes_{\rea{\chi}{\p}}Z_\chi^\p(\la).
\]
It follows from \cite[Section~6.5]{WZ1} that if $\la$ satisfies
$(\la+{}^{\tilde{\Pi}}\rho|\,\delta)^p -
(\la+{}^{\tilde{\Pi}}\rho|\,\delta)=0$, then
$Z_\chi^{\osp(1|2)}(\la)$ is reducible; hence $Z_\chi^\p(\la)$ and
so $Z_\chi^{\tilde{\Pi}}(\la)$ will be reducible. By
Proposition~\ref{prop:irred-verma}, we have
$(\la+{}^{\tilde{\Pi}}\rho|\,\delta)^p -
(\la+{}^{\tilde{\Pi}}\rho|\,\delta)$ divides
${}^{\tilde{\Pi}}\Phi(\la)$, that is,
$((\la|\,\delta)^p-(\la|\,\delta))$ divides $\Phi'$. Note that for
two such roots $\delta$ and $\delta'$,
$((\la|\,\delta)^p-(\la|\,\delta))$ and
$((\la|\,\delta')^p-(\la|\,\delta'))$ are coprime if $\delta \neq
\pm \delta'$. Since for any such root $\delta$ either $\delta$ or
$-\delta$ is in ${}^\Pi\od{\Delta}^+$, and since the arbitrary
choice of such $\delta$ we conclude that
\[
\prod_{\delta \in {}^\Pi\Delta^+_{\bar{1}}\setminus {}^\Pi
\od{\overline{\Delta}}^+}(\la|\,\delta) \cdot \prod_{2\delta \in
{}^\Pi\Delta^+_{\bar{0}}\setminus {}^\Pi
\ev{\overline{\Delta}}^+}((\la|\,2\delta)^{p-1}-1) \text{ divides }
\Phi'.
\]

Next observe that any odd root $\beta \in \od{\overline{\Delta}}$
(of type (ii)) appears in some set of simple roots. The root vector
$X_\beta$ generates an embedded $\mathfrak{sl}(1|\,1)$. Using
similar arguments as for type (iii) simple roots above, we can show
that
\[
\prod_{\beta \in {}^\Pi \od{\overline{\Delta}}^+}(\la|\,\beta)
\text{ divides }\Phi'.
\]
In the proof, we need an irreducibility criterion for
$\mathfrak{sl}(1|\,1)$-baby Verma modules, which can be easily
deduced from that for $\gl(1|\,1)$-baby Verma modules as in
\cite[Proposition~7.7]{WZ2}.

For roots in $\ev{\overline{\Delta}}$ (of type (i)), in a similar
but classical manner (cf. \cite[Proof of Proposition~6]{Rud}), we
can show that
\[
\prod_{\alpha \in
{}^\Pi\ev{\overline{\Delta}}^+}((\la|\,\alpha)^{p-1}-1) \text{
divides } \Phi'.
\]

Finally, the Proposition follows from a degree consideration and the
fact that the above three factors are mutually coprime.
\end{proof}

\begin{theorem}\label{thm:irred-verma}
A baby Verma module $Z_\chi^\Pi(\la)$ for $\la \in \La_\chi$ is
irreducible if and only if
\[
\prod_{\alpha \in {}^\Pi
\Delta^+_{\bar{0}}}((\la+{}^\Pi\rho|\,\alpha)^{p-1}-1) \cdot
\prod_{\beta \in {}^\Pi \Delta^+_{\bar{1}}} (\la+
{}^\Pi\rho|\,\beta) \neq 0.
\]
\end{theorem}
\begin{proof}
Follows readily from Propositions~\ref{prop:irred-verma} and
\ref{prop:poly'}.
\end{proof}

\begin{theorem}\label{thm:semisimplicity-rea}
The algebra $\rea{\chi}{\g}$ is a semisimple algebra if and only if
$\chi$ is regular semisimple.
\end{theorem}
\begin{proof}
The argument, which uses the irreducibility criterion in
Theorem~\ref{thm:irred-verma}, is pretty standard. We include it
here just for the sake of completeness.

Since $\chi$ satisfies $\chi(X_\alpha)=0$ for each $\alpha \in
\ev\Delta$, for any set of simple roots $\Pi$, the baby Verma
modules $Z^\Pi_\chi(\la)$ for $\la \in \La_\chi$ have unique
irreducible quotients, and they form a complete and irredundant set
of irreducible $\rea{\chi}{\g}$-modules. Now by Wedderburn Theorem
and a dimension counting argument, $\rea{\chi}{\g}$ is semisimple if
and only if all the baby Verma modules $Z^\Pi_\chi(\la)$ for $\la
\in \La_\chi$ are simple. By Theorem~\ref{thm:irred-verma},
$Z^\Pi_\chi(\la)$ being simple for all $\la \in \La_\chi$ is
equivalent to ${}^\Pi\Phi(\la)\neq 0$ for all $\la \in \La_\chi$,
which in turn is equivalent to (i) $(\la + {}^\Pi\rho)(H_\alpha)
\notin \mathbb{F}_p \setminus \{0\}$ for all $\alpha \in
\Delta_{\bar{0}}$ and (ii) $(\la+{}^\Pi\rho)(H_\beta) \neq 0$ for
all $\beta \in \Delta_{\bar{1}}$.

Recall that under current assumption, $\chi$ is regular semisimple
if and only if $\chi(H_\alpha) \neq 0$ for all $\alpha \in \Delta$.
If $\chi$ is regular semisimple, then it follows that for any $\la
\in \La_\chi$, $\la(H_\alpha) \notin \mathbb{F}_p$ for all $\alpha
\in \Delta$ since $\la(H_\alpha)^p-\la(H_\alpha)=\chi(H_\alpha)^p$.
In this situation, both (i) and (ii) are true since
${}^\Pi\rho(H_\alpha) \in \mathbb{F}_p$ for any $\alpha \in \Delta$.
Hence all $Z^\Pi_\chi(\la)$ are simple and $\rea{\chi}{\g}$ is
semisimple.

Conversely, if $\chi$ is not regular semisimple, then
$\chi(H_\alpha)=0$ for some $\alpha \in \Delta$. Let us assume
$\alpha \in \Delta_{\bar{0}}$, since the other case can be argued in
a similar fashion. Then $\la(H_\alpha) \in \mathbb{F}_p$ for $\la
\in \La_\chi$. Since shifting the value of $\la(H_\alpha)$ by a
number in $\mathbb{F}_p$ will still result in an element in
$\La_\chi$ (noting that the values of $\la(H_\beta)$ for some $\beta
\in \Delta$ will be changing correspondingly), we may thus assume
$(\la + {}^\Pi\rho)(H_\alpha) =1$. Then ${}^\Pi\Phi(\la)=0$ and
$Z^\Pi_\chi(\la)$ is reducible by Theorem~\ref{thm:irred-verma}.
Hence $\rea{\chi}{\g}$ is not semisimple.

\end{proof}

\begin{remark}\label{rem:semisimplicity}
Note that the ``if'' part of the theorem is a consequence (cf.
\cite[Corollary 5.7]{WZ1}) of the Super Kac-Weisfeiler Conjecture
(Theorem~\ref{th:KW}).

Also, for type I basic classical Lie superalgebras,
Theorems~\ref{thm:irred-verma} and \ref{thm:semisimplicity-rea} are
consequences of an equivalence of categories between typical
$\rea{\chi}{\g}$-modules and typical $\rea{\chi}{\ev\g}$-modules
(see \cite[Theorems~4.1 and 4.3]{Z}).
\end{remark}

\end{document}